\documentclass{amsart}
\usepackage{amsmath, amsthm}
\usepackage{hyperref}
\hypersetup{backref,urlcolor= black,colorlinks=true}

\usepackage{graphics}
\DeclareGraphicsExtensions{ext1,ext2,...,extn}

\usepackage{epsfig}

\title[\sc A Geometric Criterion for the Finite Generation of the Cox Ring ...]{
A Geometric Criterion for the Finite Generation of the Cox Ring of
Projective Surfaces}

\author[\sc De La Rosa Navarro]{Brenda De La Rosa Navarro}
\address{Instituto de F\'{\i}sica y Matem\'aticas (IFM)\\
Universidad Michoacana de San Nicol\'as de Hidalgo\\
Edificio C-3, Ciudad  Universitaria. C. P. 58040 Morelia\\
Michoac\'an, M\'exico} \email{brenda@ifm.umich.mx}

\author[\sc Lahyane]{Mustapha Lahyane}
\address{Instituto de F\'{\i}sica y Matem\'aticas (IFM)\\
Universidad Michoacana de San Nicol\'as de Hidalgo\\
Edificio C-3, Ciudad  Universitaria. C. P. 58040 Morelia\\
Michoac\'an, M\'exico} \email{lahyane@ifm.umich.mx}

\author[\sc Moreno Mej\'{\i}a]{Israel Moreno Mej\'{\i}a}
\address{Instituto de Matem\'aticas\\
Universidad Nacional Aut\'onoma de M\'exico\\
\'Area de la Investigaci\'on Cient\'{\i}fica, Circuito Exterior,
Ciudad Universitaria, Coyoac\'an\\
C.P. 04510, M\'exico D.F.\\
M\'exico} \email{israel@matem.unam.mx}

\author[\sc Osuna Castro]{Osvaldo Osuna Castro}
\address{Instituto de F\'{\i}sica y Matem\'aticas (IFM)\\
Universidad Michoacana de San Nicol\'as de Hidalgo\\
Edificio C-3, Ciudad  Universitaria. C. P. 58040 Morelia\\
Michoac\'an, M\'exico} \email{osvaldo@ifm.umich.mx}


\newtheoremstyle{theorem}
  {10pt}          
  {10pt}  
  {\sl}  
  {\parindent}     
  {\bf}  
  {. }    
  { }    
  {}     
\theoremstyle{theorem}
\newtheorem{theorem}{Theorem}
\newtheorem{corollary}[theorem]{Corollary}
\newtheorem{lemma}[theorem]{Lemma}

\newtheoremstyle{defi}
  {10pt}          
  {10pt}  
  {\rm}  
  {\parindent}     
  {\bf}  
  {. }    
  { }    
  {}     
\theoremstyle{defi}
\newtheorem{definition}[theorem]{Definition}

\begin{document}

\maketitle

\begin{abstract}
The aim is to give a geometric characterization of the finite
generation of the Cox ring of anticanonical rational surfaces. This
characterization is encoded in the finite generation of the
effective monoid. Furthermore, we prove that in the case of a smooth
projective rational surface having a negative multiple of its
canonical divisor with only two linearly independent global sections
(e.g., an elliptic rational surface), the finite generation is
equivalent to the fact that there are only a finite number  of
smooth projective rational curves of self-intersection $-1$. The
ground field is assumed to be algebraically closed of arbitrary
characteristic.
\end{abstract}

\section{Introduction}

In \cite{GM}, Galindo and Monserrat characterize the smooth
projective surfaces $Z$ defined over an algebraically closed field
$k$ with finitely generated Cox rings (see the next paragraph for
the definition) by means of the finiteness  of the set of integral
curves on $Z$ of negative self-intersection and the existence of a
finitely generated $k-$algebra containing two $k-$algebras
associated  naturally  to $Z$, see \cite[Theorem 1, page 94]{GM}.
The aim of this work is to give an equivalent characterization of
the finite generation of the Cox ring totally based on the geometry
of the surface and to apply the criterion to some  classes of smooth
projective rational surfaces, e.g. the anticanonical ones (i.e.,
those rational surfaces holding an effective anticanonical divisor)
and the surfaces constructed  in \cite{CPR1}, \cite{CPR2},
\cite{GM1} and \cite{GM2}; establishing thus the geometric nature of
our characterization. For some purely algebraic features of the Cox
ring of a variety, see \cite{EKW}.

Following Hu and Keel \cite{HK}, the Cox  (or the total coordinate)
ring of a smooth projective variety $V$ defined over an
algebraically closed field $k$ is the $k-$algebra defined as
follows:

$$Cox(V)=\bigoplus_{(n_{1}, \ldots, n_{r}) \in {\mathbb Z}^{r}} H^{0}(V, {\mathcal{O}}(L_{1}^{n_{1}} \otimes
\ldots \otimes L_{r}^{n_{r}})).$$

Here $(L_{1}, \ldots, L_{r})$ is a basis of the $\mathbb Z$-module
$Pic(V)$ of classes of invertible sheaves on $V$ modulo isomorphisms
under the tensor product, and we have assumed that the linear and
numerical equivalences on the group of Cartier divisors on $V$ are
the same, such assumption is satisfied for example for the smooth
projective rational surfaces $V$.

An interesting (but still) open  problem is to classify
theoretically and/or effectively and constructively all smooth
projective rational surfaces $S$ for which the $k-$algebra $Cox(S)$
is finitely generated. Masayoshi Nagata (see \cite{MN}) showed that
the surface $Z$ obtained by blowing up of the projective plane $\Bbb
P^{2}$ at nine or more points in general position has an infinite
number of $(-1)-$curves (see also  \cite{L4}, \cite{L0}, \cite{L2},
\cite{LH}, \cite{R}, \cite{SM}, \cite{UPRM} and \cite{L3}  for cases
when the points need not be in general position), consequently its
Cox ring $Cox(Z)$ is not finitely generated. Here a $(-1)-$curve on
$Z$ means a smooth projective curve on $Z$ of self-intersection
equal to $-1$. Note that in this example, the effective monoid
$M(Z)$ of $Z$ is also not finitely generated, where $M(Z)$ stands
for the set of elements of the Picard group $Pic(Z)$ of $Z$ having
at least a nonzero global section.

In this paper we mainly look for those smooth projective rational
surfaces $S$  for which the finite generation of $Cox(S)$  is
equivalent to the finite generation of $M(S)$. Our two main results,
Theorems \ref{main} and \ref{main2} below which are derived from
Theorem \ref{criterion}, give a partial answer. By the way, we have
been informed by a referee that in the characteristic zero case,
Theorem \ref{criterion} was obtained in \cite{AHL} using a different
approach.

\begin{theorem}
\label{main} Let $S$ be a smooth projective rational surface defined
over an algebraically closed field $k$ of arbitrary characteristic
such that the invertible sheaf associated to the divisor $-K_{S}$
has a nonzero global section.

The following assertions are equivalent:
\begin{enumerate}
\item $Cox(S)$ is finitely generated.
\item  $M(S)$ is finitely generated.
\item $S$ has only a finite number of $(-1)$-curves and only a
finite number of $(-2)$-curves.
\end{enumerate}
Here $K_{S}$ denotes a canonical divisor on $S$.
\end{theorem}
\begin{proof}
It follows from Lemma \ref{fg} and Theorem \ref{criterion} below.
\end{proof}

As consequences, the following two results hold:
\begin{corollary}
The Cox ring of a smooth projective rational surface having a
canonical divisor of self-intersection larger than or equal to zero
is finitely generated if and only if the set of $(-1)$-curves is
finite.
\end{corollary}
\begin{proof}
Apply Theorem \ref{main} and \cite[Proposition 4.3 (a), page 9]{LH}.
\end{proof}

\begin{corollary}
\label{larger than zero} The Cox ring of a smooth projective
rational surface having a canonical divisor of self-intersection
larger than zero is finitely generated.
\end{corollary}
\begin{proof}
Apply Theorem \ref{main} and \cite[Proposition 4.3 (a), page 9]{LH}.
\end{proof}

In particular, since a Del Pezzo surface is nothing but a blow up of
the projective plane at $r$  points with $r\leq 8$, we recover the
well known result, see \cite{BP}:
\begin{corollary}
The Cox ring of a Del Pezzo surface is finitely generated.
\end{corollary}

\begin{corollary}
The Cox ring of a smooth projective rational surface having an
integral curve algebraically equivalent to an anti-canonical divisor
is finitely generated if and only if the set of $(-2)$-curves is
finite and spans a linear subspace in the Picard group of
codimension one.
\end{corollary}
\begin{proof}
The result follows from \cite{Harb1} and  Theorem \ref{main}.
\end{proof}


Here is our second result:

\begin{theorem}
\label{main2} Let $Z$ be a smooth projective rational surface
defined over an algebraically closed field $k$ of arbitrary
characteristic such that the invertible sheaf associated to the
divisor $-rK_{Z}$ has only two linearly independent global sections
for some positive integer $r$.

The following assertions are equivalent:
\begin{enumerate}
\item $Cox(Z)$ is finitely generated.
\item  The set of smooth projective rational curves of self-intersection $-1$ on
$Z$ is finite.
\end{enumerate}
Here $K_{Z}$ denotes a canonical divisor on $Z$.
\end{theorem}
\begin{proof}
Apply Theorem \ref{criterion} below  and the fact that the set of
$(-2)$-curves  on $Z$ is finite.
\end{proof}

\section{Preliminaries}
\label{preliminaries}
\subsection{General Notions}
Let $S$ be a smooth projective surface defined over an algebraically
closed field of arbitrary characteristic.  A canonical divisor on
$S$, respectively  the Picard group  of $S$ will be denoted by
$K_{S}$ and $Pic(S)$ respectively. There is an intersection form on
$Pic(S)$ induced by the intersection of divisors on $S$, it will be
denoted by a dot, that is, for $x$ and $y$ in $Pic(S)$, $x.y$ is the
intersection number of $x$ and $y$ (see \cite{HA} and \cite{BPV}).

The following result known as the Riemann-Roch theorem for smooth
projective surfaces  is stated using the Serre duality.

\begin{lemma}
\label{RR} Let $D$ be a divisor on a smooth projective  surface $S$
having an algebraically closed field of arbitrary characteristic as
a ground field. Then the following equality holds:
$$h^{0}(S, O_{S}(D))- h^{1}(S, O_{S}(D)) +
h^{0}(S, O_{S}(K_{S}-D))= 1+ p_{a}(S)+ \frac{1}{2}(D^{2}-D.K_{S}).$$
$O_{S}(D)$ (respectively, $p_{a}(S)$) being an invertible sheaf
associated canonically to the divisor $D$ (respectively, the
arithmetic genus of $S$, that is $\chi (O_{S})-1$, where $\chi$ is
the Euler characteristic function).
\end{lemma}

Here we recall some standard results, see \cite{Harb3}, \cite{HA}
and \cite{BPV}. A divisor class  modulo linear equivalence  $x$ of a
smooth projective surface $S$  is effective, respectively
numerically effective (nef in short)  if an element of $x$ is an
effective, respectively numerically effective, divisor on $S$. Here
a divisor $D$ on $S$ is nef if  $D.C\geq 0$ for every integral curve
$C$ on $S$.  Now, we start with some properties which follow from a
successive iterations of blowing up closed points of a smooth
projective rational surface.

\begin{lemma} Let $ \pi ^{\star}: NS(X) \rightarrow NS(Y)$ be
the natural group homomorphism on N\'eron-Severi  groups induced by
a given birational morphism $ \pi : Y \rightarrow X$ of smooth
projective rational surfaces. Then $\pi ^{\star}$ is an injective
intersection-form preserving map of free abelian groups of finite
rank. Furthermore, it  preserves the dimensions of cohomology
groups, the effective divisor classes and the numerically effective
divisor classes.
\end{lemma}
\begin{proof}
See \cite[Lemma II.1, page 1193]{Harb5}.
\end{proof}

\begin{lemma}
\label{effcneffne}
 Let $x$ be an element of the N\'eron-Severi group $NS(X)$
of a smooth projective rational surface $X$. The effectiveness or
the the nefness  of $x$ implies the noneffectiveness of $k_{X}-x$,
where $k_{X}$ denotes the element of $Pic(X)$ which contains a
canonical divisor on $X$. Moreover, the nefness of $x$ implies also
that the self-intersection of $x$ is greater than or equal to zero.
\end{lemma}
\begin{proof}
See \cite[Lemma II.2, page 1193]{Harb5}.
\end{proof}

The following result is also needed. We recall that a $(-1)$-curve,
respectively a $(-2)$-curve, is a smooth rational curve of
self-intersection $-1$, respectively $-2$.

\begin{lemma}
\label{fg}The monoid of effective divisor classes modulo linear
equivalence on a smooth projective rational surface $X$ having an
effective anticanonical divisor is finitely generated if and only if
$X$ has only a finite number of $(-1)$-curves and only a finite
number of $(-2)$-curves.
\end{lemma}
\begin{proof}
See \cite[Corollary 4.2, page 109]{LH}.
\end{proof}

\subsection{Extremal Surfaces}
Let $Nef(S)$ denotes the set of  nef elements in the Picard group
$Pic(S)$ of a smooth projective surface $S$, it has obviously an
algebraic structure as a monoid. We define two more submonoids
$Char(S)$ and $[Char(S):Nef(S)]$  of $Pic(S)$  (see \cite{CG} and
\cite{GM}) as follows:

\begin{definition}
With notation as above.
\begin{enumerate}
\item The {\it characteristic monoid} $Char(S)$ of $S$ is the set of
elements $x$ in $Pic(S)$ such that there exists an effective divisor
on S whose associated complete linear system is base point free and
whose class in $Pic(S)$ is equal to $x$.

\item The {\it monoid of fractional base point free effective
classes} $[Char(S):Nef(S)]$ of $S$  is the set of elements $y$ in
$Pic(S)$ such that there exists a positive integer $n$ with $ny \in
Char(S)$.

\end{enumerate}
\end{definition}

The main properties that we are interested in regarding $Char(S)$
and $[Char(S):Nef(S)]$ of a smmoth projective surface are the ones
in the Lemma below. Their proofs are straightforward.

\begin{lemma}
With notation as above, the followings hold:
\begin{enumerate}
\item $Char(S)$ and $[Char(S):Nef(S)]$ are submonoids of $Nef(S)$.
\item $Char(S) \subseteq [Char(S):Nef(S)]$.
\end{enumerate}

\end{lemma}

Here we define the ingredient needed for our criterion:

\begin{definition}
With notation as above, $S$ is extremal if the monoid of fractional
base point free effective classes $[Char(S):Nef(S)]$ is maximal,
that is, if $Nef(S)=[Char(S):Nef(S)]$.
\end{definition}

\section{The Criterion}
Now we are able to state our geometric criterion:

\begin{theorem}
\label{criterion} Let $S$ be a smooth projective surface defined
over an algebraically closed field of arbitrary characteristic. The
following assertions are equivalents:
\begin{enumerate}

\item The Cox ring $Cox(S)$ is finitely generated.

\item $S$ satisfies the following two properties:
\begin{itemize}
\item[i.]  $S$ is extremal, and
\item[ii.] the effective monoid $M(S)$ of $S$ is finitely generated.
\end{itemize}

\item $S$ satisfies the following two properties:
\begin{itemize}
\item[i.] $S$ is extremal, and
\item[ii.] the nef monoid $Nef(S)$ of $S$ is finitely generated.
\end{itemize}
\end{enumerate}

\end{theorem}
\begin{proof} By duality,  it is obvious that $(2)$ is equivalent to
$(3)$. Assume that the Cox ring $Cox(S)$ of $S$ is finitely
generated, it follows that the effective monoid $M(S)$ of $S$ is
finitely generated. On the other hand, if $y$ is an effective
element of $Nef(S)$, then  if $y$ belongs to $Char(S)$, we are done
and if not we may find some positive integer $s$ such that $sy$ is
an element of $Char(S)$. Hence $[Char(S):Nef(S)]$, i.e., $S$ is
extremal. Conversely, if  $S$ is extremal, and the nef monoid
$Nef(S)$ of $S$ is finitely generated, then it follows from
\cite{GM} that the Cox ring of $S$ is finitely generated.
\end{proof}


\section*{Acknowledgements}
 The first three authors were supported
by PAPIIT IN102008  research grant and  the projects CIC-UMSNH 2010
and 2011.


\begin{thebibliography}{99}

\bibitem{AHL} M. Artebani, J. Hausen, A. Laface {\em On Cox rings of K3 surfaces},
Compos. Math. vol. 146 (2010), no 4.  964-998.

\bibitem {BPV} W. Barth, C. Peters  , A. Van de Ven.
\textit{Compact Complex Surfaces.} Berlin, Springer  (1984).
%
\bibitem{BP} V. Batyrev, O. Popov, {\em The Cox ring of a Del Pezzo
surface}, in: B. Poonen, Y. Tschinkel (Eds.), Arithmetic of
Higher-Dimensional Algebraic Varieties, in: Progr. Math., vol. 226,
Birkh$\ddot{a}$user, 2004.
%
%
%
%

\bibitem{CG} A. Campillo, G. Gonz\'alez-Sprinberg, {\em On characteristic cones, clusters and chains
of infinitely near points}, in: Brieskorn Conference Volume, in:
Progr. Math., vol. 162, Birkh$\ddot{a}$user, 1998.

%
\bibitem{CPR1} A. Campillo, O. Piltant, A. Reguera,
{\em Cones of Curves and of Line Bundles on Surfaces Associated with
Curves Having one Point at Infinity}, Proceedings of the London
Mathematical Society, Vol.84, Issue 03 (2002), 559-580.

%
\bibitem{CPR2}  A. Campillo, O. Piltant, A. Reguera,
{\em Cones of curves and of line bundles at infinity}. Journal of
Algebra  293, 503-542 (2005).
%

\bibitem{EKW} E.J. Elizondo, K. Kurano,  K. Watanabe,  {\em The total coordinate ring
of a normal projective variety,} Journal of Algebra, Vol. 276 (2)
(2004), pp. 625-637.

%
\bibitem{GM1} C. Galindo, F. Monserrat,
{\em The cone of curves associated to a plane configuration},
Commentarii Mathematici Helvetici 80 (2005), 75-93.
\bibitem{GM2}    C. Galindo, F. Monserrat,
{\em The cone of curves of Line Bundles of a Rational Surface,}
International Journal of Mathematics. Vol. 15, No. 4 (2004),
393-407.

%
\bibitem{GM} C. Galindo, F. Monserrat,  {\em The total coordinate ring
of a smooth projective surface,} Journal of Algebra, Vol. 284
(2005), pp. 91-101.
%
\bibitem{Harb1}  B. Harbourne,  {\em Blowings-up of $\Bbb P ^{2}$ and their blowings-down,}
Duke Mathematical Journal {\bf 52}:1 (1985), 129-148.
%
%
\bibitem{Harb3} B. Harbourne,  {\em Complete linear systems
on rational surfaces,} Transactions of the American Mathematical
Society, Vol. 289, No. 1. (May, 1985), pp. 231-226.
%

%
%


%
%
\bibitem{Harb5} B. Harbourne, \textit{Anticanonical rational surfaces},
Transactions  of the American Mathematical Society, Volume {\bf 349}
(1997), Number 3, 1191-1208.
%
\bibitem{HA} R. Hartshorne, \textit{Algebraic Geometry}, Graduate Texts in
Mathematics, Springer Verlag (1977).
%
%
%
%
\bibitem{HK} Y.  Hu, S. Keel,    {\em  Mori Dream Spaces and GIT},
Michigan Math. J. {\bf 48} (2000) 331-348.
%



%
\bibitem{L0}  M. Lahyane,    {\em  Exceptional curves on rational surfaces
having $K^{2}\geq 0$}. C. R. Acad. Sci. Paris, Ser. I {\bf 338}
(2004) 873-878.
%
\bibitem{L2}  M. Lahyane,  {\em Rational surfaces having only a
finite number of exceptional curves}. Mathematische Zeitschrift
Volume {\bf 247}, Number 1, 213-221 (May 2004).
%
\bibitem{L3} M. Lahyane, {\em Exceptional curves on smooth rational surfaces
with $-K$ not nef and of self-intersection zero}. Proceedings of the
American  Mathematical  Society Volume {\bf 133}, Number 6 (2005)
1593-1599.
%
\bibitem{L4} M. Lahyane, {\em Irreducibility of the $(-1)$-classes on smooth rational surfaces}.
Proceedings of the American  Mathematical  Society  Volume {\bf
133}, Number 8 (2005) 2219-2224.
%
\bibitem{LH} M. Lahyane, B. Harbourne, {\em  Irreducibility of $-1$-classes
on anticanonical rational surfaces and finite generation of the
effective monoid}. Pacific Journal of Mathematics  Volume {\bf 218},
Number 1 (2005), pp. 101-114.
%

%
%
\bibitem{UPRM} R. Miranda, U. Persson,  {\em On extremal rational
elliptic surfaces}. Mathematische Zeitschrift {\bf 193,} 537-558
(1986).
%
\bibitem{SM} S. Mori, {\em Threefolds whose canonical bundles are not numerically effective},
Annals of Mathematics (2),  \textbf{116}  (1982),  no. 1, 133--176.
%
\bibitem{MN} M. Nagata, {\em On rational surfaces, II},  Memoirs of the
College of Science, University of Kyoto, Series A \textbf{33}
(1960), no.~2, 271--293.
%
\bibitem{R} J. Rosoff, {\em Effective divisor classes and blowings-up of $\mathbb
P^{2}$}, Pacific Journal  of Mathematics, Volume 89, Number 2(1980),
pp. 419--429.

\end{thebibliography}
\end{document}